\documentclass{amsart}
\usepackage{graphicx}
\usepackage{amsfonts}
\usepackage{float} 
\newtheorem{tm}{Theorem}

\newtheorem{defi}{Definition}

\newtheorem{rem}{Remark}
\newtheorem{rems}{Remarks}
\newtheorem{lm}{Lemma}
\newtheorem{ex}{Example}

\newtheorem{nota}{Notation}

\newtheorem{op}{Open problem}

\begin{document}

\title{Univariate polynomials and the contractibility of certain sets}
\author{Vladimir Petrov Kostov}
\address{Universit\'e C\^ote d’Azur, CNRS, LJAD, France} 
\email{vladimir.kostov@unice.fr}

\begin{abstract}
  We consider the set $\Pi ^*_d$ of monic polynomials
  $Q_d=x^d+\sum _{j=0}^{d-1}a_jx^j$, $x\in \mathbb{R}$, $a_j\in \mathbb{R}^*$,
  having $d$ distinct real roots, and its subsets defined by fixing the signs
  of the coefficients $a_j$. We show that for every choice of these signs,
  the corresponding subset is
  non-empty and contractible. A similar result holds true in the cases of
  polynomials $Q_d$ of even degree $d$ and having no real roots or
  of odd degree and having exactly one real root. For even $d$ and when $Q_d$
  has exactly two real roots which are of opposite signs, the subset is
  contractible.
  For even $d$ and when $Q_d$ has two positive (resp. two negative) roots,
  the subset is contractible or empty. It is empty
  exactly when the constant term is positive, among the other even coefficients
  there is at least one which is negative, and all odd coefficients are
  positive (resp. negative).\\ 


\end{abstract}
\maketitle

\section{Introduction}

In the present paper we consider the general family of real monic
univariate polynomials $Q_d=x^d+\sum _{j=0}^{d-1}a_jx^j$. It is a classical fact
that the subsets of $\mathbb{R}^d\cong Oa_0\ldots a_{d-1}$ of values of the
coefficients $a_j$ for which the polynomial $Q_d$ has one and the same number
of distinct real roots are contractible open sets. These sets are the
$[d/2]+1$ open
parts of $R_{1,d}:=\mathbb{R}^d\setminus \Delta _d$, where $\Delta _d$
is the {\em discriminant set} corresponding to the family $Q_d$.

\begin{rems}\label{remsdiscrim}
  {\rm (1) One defines the discriminant set by the two conditions:

    (a) The set $\Delta ^1_d$ is defined by the equality Res$(Q_d,Q_d',x)=0$,
    where Res$(Q_d,Q_d'$, $x)$ is the
resultant of the polynomials $Q_d$ and $Q_d'$, i.~e. the determinant of the
corresponding Sylvester matrix.

(b) One sets $\Delta _d:=\Delta ^1_d\setminus \Delta ^2_d$, where  
$\Delta ^2_d$
is the set of values of
the coefficients $a_j$ for which there is a multiple complex conjugate
pair of roots of $Q_d$ and no multiple real root.

One observes that dim$(\Delta _d)=$dim$(\Delta ^1_d)=d-1$ and
dim$(\Delta ^2_d)=d-2$. Thus $\Delta _d$ is the set of values of
$(a_0,\ldots ,a_{d-1})$ for which the polynomial $Q_d$ has a multiple real root.

(2) The discriminant set is invariant under the one-parameter group of
  quasi-homogeneous dilatations $a_j\mapsto u^{d-j}a_j$, $j=0$, $\ldots ,~d$.}
  \end{rems}

\begin{rem}
{\rm If one considers the subsets of $\mathbb{R}^d$ for which the polynomial
$Q_d$ has one and the same numbers of positive and negative roots (all of
them distinct) and no zero roots, then these sets will be the open parts
of the set $R_{2,d}:=\mathbb{R}^d\setminus (\Delta _d\cup \{ a_0=0\} )$.
To prove their
connectedness one can consider the mapping ``roots $\mapsto$ coefficients''.
Given two sets of nonzero roots with the same numbers of negative and
positive roots
(in both cases they are all simple) one can continuously deform the first set
into the second one while keeping the absence of zero roots, the
numbers of positive and negative roots and
their simplicity throughout the deformation. The existence of this deformation
implies the existence of a continuous path in the set $R_{2,d}$ connecting the
two polynomials $Q_d$ with the two sets of roots.}
\end{rem}

In the present text we focus on polynomials without vanishing coefficients
and we consider the set

$$R_{3,d}:=\mathbb{R}^d\setminus (\Delta _d\cup \{ a_0=0\}
\cup \{ a_1=0\} \cup \cdots
\cup \{ a_{d-1}=0\} )~.$$
We discuss the question when its subsets corresponding to given
numbers of positive and negative roots of $Q_d$ and to given signs of its
coefficients are contractible.

\begin{nota}\label{notaelliptic}
  {\rm (1) We denote by $\sigma$ the $d$-tuple
    $($sign$(a_0), \ldots$, sign$(a_{d-1}))$, where sign$(a_j)=+$ or~$-$,
    by $\mathcal{E}_d$ the set 
  of {\em elliptic} polynomials $Q_d$, i.~e.
  polynomials with no real roots (hence $d$ is even and $a_0>0$), and by
  $\mathcal{E}_d(\sigma )\subset \mathcal{E}_d$ the set consisting of elliptic
  polynomials $Q_d$ with signs of the coefficients defined by
  $\sigma$.

  (2) For $d$ odd and for a given $d$-tuple $\sigma$, we denote by
  $\mathcal{F}_d(\sigma )$ the set of monic real polynomials $Q_d$
  with signs of their coefficients defined by the
  $d$-tuple $\sigma$ and having exactly one real (and simple) root.

  (3) For $d$ even, we denote by $\mathcal{G}_d(\sigma )$
  the set of polynomials $Q_d$
  having signs of the coefficients defined by the $d$-tuple $\sigma$ and
  having exactly two simple real roots.}
\end{nota}

\begin{rem}\label{remelliptic}
  {\rm For an elliptic polynomial $Q_d$, one has $a_0>0$, because for
    $a_0<0$, there is at least one positive root. The sign of the
    real root of a polynomial
    of $\mathcal{F}_d(\sigma )$ is opposite to sign$(a_0)$. A polynomial from
    $\mathcal{G}_d(\sigma )$ has two roots of same (resp. of opposite) signs
  if $a_0>0$ (resp. if $a_0<0$).}
  \end{rem}

In order to formulate our first result we need the following definition:

\begin{defi}\label{defispecial}
  {\rm (1) For $d$ even and $a_0<0$, we set
    $\mathcal{G}_{d,(1,1)}(\sigma ):=\mathcal{G}_d(\sigma )$. For $d$ even and 
    $a_0>0$, we set
    $\mathcal{G}_d(\sigma ):=\mathcal{G}_{d,(2,0)}(\sigma )
    \cup \mathcal{G}_{d,(0,2)}(\sigma )$,
    where for $Q_d\in \mathcal{G}_{d,(2,0)}$ (resp. $Q_d\in \mathcal{G}_{d,(0,2)}$),
    $Q_d$ has two positive (resp. two negative) distinct roots and
    no other real roots. Clearly $\mathcal{G}_{d,(2,0)}(\sigma )
    \cap \mathcal{G}_{d,(0,2)}(\sigma )=\emptyset$. 

    (2) For $d$ even, we define two special cases according to the signs of
    the coefficients of $Q_d$
    and the quantities of its positive or negative real roots:
    \vspace{1mm}
    
    {\em Case 1).} The constant term and all coefficients
    of monomials of odd degrees are positive, there is at least one coefficient
    of even degree which is negative, and $Q_d$ has $2$ positive
    and no negative roots.
    \vspace{1mm}

{\em Case 2).} The constant term is positive, all coefficients
    of monomials of odd degrees are negative, there is at least one coefficient
    of even degree which is negative, and $Q_d$ has $2$ negative
    and no positive roots.
    \vspace{1mm}

    Cases 1) and 2) are exchanged when one performs the
    change of variable $x\mapsto -x$.}
\end{defi}

Our first result concerns real polynomials
with not more than $2$ real roots:

\begin{tm}\label{tm2roots}
(1) For $d$ even and for each $d$-tuple $\sigma$, the subset
  $\mathcal{E}_d(\sigma )\subset \mathcal{E}_d$ is non-empty and convex hence
  contractible.

  (2) For $d$ odd and for each $d$-tuple $\sigma$, the set
  $\mathcal{F}_d(\sigma )$ is non-empty and contractible.

  (3) For $d$ even and for each $d$-tuple $\sigma$ with $a_0<0$, the set
  $\mathcal{G}_{d,(1,1)}(\sigma )$ is contractible. For $d$ even and for each
  $d$-tuple $\sigma$ with $a_0>0$, each set
  $\mathcal{G}_{d,(2,0)}(\sigma )$ (resp. $\mathcal{G}_{d,(0,2)}(\sigma )$)
  is contractible or empty. It is empty
  exactly in Case 1) (resp. in Case~2)). 
\end{tm}

The theorem is proved in Section~\ref{secprtm2roots}. The next result of this
paper concerns {\em hyperbolic polynomials},
i.~e. polynomials
$Q_d$ with $d$ real roots counted with multiplicity.

\begin{nota}\label{notaPid}
  {\rm We denote by $\Pi _d$
the {\em hyperbolicity domain}, i.~e. the subset of $\mathbb{R}^d$ for which
the corresponding polynomial $Q_d$ is hyperbolic. The interior of $\Pi _d$
is the set of polynomials having $d$ distinct real roots and its border 
$\partial \Pi _d$ equals $\Delta _d\cap \Pi _d$. We set}

  $$\Pi _d^*:=\Pi _d\setminus (\Delta _d\cup \{ a_0=0\} \cup \{ a_1=0\}
  \cup \cdots \cup \{ a_{d-1}=0\} )~.$$
       {\rm Thus $\Pi _d^*$ is the set of monic degree $d$ univariate
         polynomials with $d$ distinct real roots and with all
         coefficients non-vanishing. We denote by $\Pi _d^k$ and $\Pi _d^{*k}$
         the projections of the sets $\Pi _d$ and $\Pi _d^*$ 
         in the space $Oa_{d-k}\ldots a_{d-1}$ (hence $\Pi _d^d=\Pi _d$ and
         $\Pi _d^{*d}=\Pi _d^*$),
         by $\partial \Pi _d^k$ the border of $\Pi _d^k$ and by $pos$ and $neg$
         the numbers of positive and
         negative roots of a polynomial $Q_d$ having no vanishing coefficients.

      We set $a:=(a_0,a_1,\ldots ,a_{d-1})$, $a':=(a_1,\ldots ,a_{d-1})$, 
  $a'':=(a_2,\ldots ,a_{d-1})$ and $a^{(k)}:=(a_k,\ldots ,a_{d-1})$.
  In what follows we
  use the same notation for functions and for their graphs.}
\end{nota}

\begin{rems}\label{remsDescartes}
  {\rm (1) For a hyperbolic polynomial with no vanishing coefficients,
    the $d$-tuple $\sigma$ defines the
    numbers $pos$ and $neg$. Indeed, by Descartes' rule of signs a real
    univariate polynomial $Q_d$ with $c$ sign changes in its sequence of
    coefficients
    has $\leq c$ positive roots and the difference $c-pos$ is even,
    see \cite{VJ} and~\cite{Fo}. When
    applying this rule to the polynomial $Q(-x)$ one finds that the number
    $p$ of sign preservations is $\geq neg$ and the difference $p-neg$ is even.
    For a hyperbolic polynomial one has $pos+neg=c+p=d$, so in this
    case $c=pos$ and $p=neg$.

    (2) By Rolle's theorem the non-constant derivatives of a hyperbolic
    polynomial (resp. of a polynomial of the set $\Pi _d^*$)
    are also hyperbolic (resp. are hyperbolic with all roots non-zero and
    simple).
    Hence for two hyperbolic polynomials of
    the same degree and with
    the same signs of their respective coefficients, their derivatives of the
    same orders have one and the same numbers of positive and negative roots.}
  \end{rems}

Our next result is the following theorem
(proved in Section~\ref{secprtmconnecthyp}):

\begin{tm}\label{tmconnecthyp}
For each $d$-tuple $\sigma$, there exists exactly one
  open component of the set $\Pi _d^*$ the polynomials $Q_d$ from which have
  exactly $pos$ positive simple and $neg$ negative simple roots and have
  signs of the coefficients as defined by~$\sigma$. This component is
  contractible. 
\end{tm}

One can give more explicit information about the components of the
set $\Pi _d^*$. 
Denote by $M$ such a component defined after a $d$-tuple
$\sigma$ and by $M^k$ its projection
in the space $Oa_{d-k}\cdots a_{d-1}$. It is shown in~\cite{KoSe} (see
Proposition~1 therein) that $M$ is non-empty. In
Section~\ref{secprtmconnecthyp} we prove the
following statement:

\begin{tm}\label{tmconcrete}
  For $k\geq 3$, the set $M^k$ is the set of all points between the graphs
  $L^k_{\pm}$ of two
  continuous functions defined on $M^{k-1}$:

  $$M^k=\{ a^{(d-k)}\in \mathbb{R}^{d-k}~|~
  L^k_-(a^{(d-k+1)})<a_{d-k}<L^k_+(a^{(d-k+1)}),~a^{(d-k+1)}\in M^{k-1}\} ~.$$
  The functions $L^k_{\pm}$ can be extended
  to continuous functions defined on $\overline{M^{k-1}}$, 
  whose values might
  coincide (but this does not necessarily happen) only on $\partial M^{k-1}$.
\end{tm}

\begin{rem}
%

  {\rm Theorem~\ref{tmconnecthyp} can be 
    deduced from Theorem~\ref{tmconcrete} 
    (but we give in Section~\ref{secprtmconnecthyp} a direct proof which is
    short enough).
    Indeed, given a component $M$ of the set $\Pi _d^*$, 
one can successively contract it into its projections $M^{d-1}$, $M^{d-2}$,
$\ldots$, $M^2$. The latter is
one of the sets $\Pi ^{*2}_{d\pm \pm}$ defined in Example~\ref{exk2} which
are contractible.}
  \end{rem}

In Section~\ref{secknown} we remind some results which are used in the proof of
Theorem~\ref{tmconnecthyp}. In Section~\ref{secex} we introduce some notation
and we give examples concerning the sets $\Pi _d$ and $\Pi _d^*$ for
$d=1$, $2$ and~$3$. These examples are used in the proofs of
Theorems~\ref{tmconnecthyp} and~\ref{tmconcrete}. In 
Section~\ref{seccomments} we make comments on Theorems~\ref{tm2roots}, 
\ref{tmconnecthyp} and~\ref{tmconcrete} 
and we formulate open problems.

\section{Known results about the hyperbolicity domain\protect\label{secknown}}

Before proving Theorems~\ref{tm2roots}, \ref{tmconnecthyp} and~\ref{tmconcrete} 
we remind some results about the
set $\Pi _d$ which
are due to V.~I.~Arnold, A.~B.~Givental and the author, see \cite{Ar},
\cite{Gi} and \cite{KoProcRSE} or Chapter~2 of
\cite{Ko} and the references therein. 

\begin{nota}\label{notaViete}
  {\rm We denote by $K_d$ the simplicial angle
    $\{ x_1\geq x_2\geq \cdots \geq x_d\}\subset \mathbb{R}^d$ and by
    $\tilde{\mathcal{V}}$ the
    Vi\`ete mapping}
  $$\tilde{\mathcal{V}}:(x_1,\ldots ,x_d)\mapsto
  (\varphi _1,\ldots ,\varphi _d)~,~~~\, \,
  \varphi _j=\sum _{1\leq i_1<i_2<\cdots <i_j\leq d}x_{i_1}x_{i_2}\cdots x_{i_j}~.$$
          {\rm {\em Strata} of $K_d$ are denoted by their
            {\em multiplicity vectors}. E.~g. for $d=5$, the stratum of $K_5$
            defined by the multiplicity vector $(2,2,1)$ is the set
            $\{ x_1=x_2>x_3=x_4>x_5\}\subset \mathbb{R}^5$. The same notation
            is used for strata of $\Pi _d$ which is justified by parts (3) and
            (4) of Theorem~\ref{tmremind}.}
  \end{nota}

\begin{rem}\label{remstrata}
  {\rm The set $\Delta _d\cap \Pi _d=\Delta ^1_d\cap \Pi _d$ consists of points
    $a\in \Pi _d\subset \mathbb{R}^d$, for which the hyperbolic 
    polynomial $Q_d$ has at least one root of multiplicity $\geq 2$. That is
    why $\Pi _d\setminus \Delta _d=\Pi _d\setminus \Delta ^1_d=S_{1^d}$ is the
    stratum of $\Pi _d$ with multiplicity vector $1^d=(1,\ldots ,1)$ and}

    $$\Pi _d^*=S_{1^d}\setminus (\{ a_0=0\} \cup \cdots \cup \{ a_{d-1}=0\} )~.$$
  {\rm The strata of $\Pi _d^*$
    (they are all of dimension $d$,
    so they can also be called {\em components}) are of the form}

  $$S_{1^d}(\sigma ):=\{ a\in S_{1^d}~|~{\rm sign}(a_j)=
  \sigma _j, 0\leq j\leq d-1\}$$
  {\rm for some $\sigma =(\sigma _0,\ldots ,\sigma _{d-1})\in \{ \pm \}^d$.}
  \end{rem}

\begin{tm}\label{tmremind}
  (1) For $k\geq 3$, every non-empty fibre $\tilde{f}_k$ of the projection
  $\pi ^k:\Pi _d^k\rightarrow \Pi _d^{k-1}$ is either a segment or a point.

  (2) The fibre $\tilde{f}_k$ is a segment (resp. a point) exactly
  if the fibre is over a
  point of the interior of $\Pi _d^{k-1}$ (resp. over $\partial \Pi _d^{k-1}$).

  (3) The mapping $\tilde{\mathcal{V}}:K_d\rightarrow \Pi _d$
  is a homeomorphism.

  (4) The restriction of the mapping $\tilde{\mathcal{V}}$
  to (the closure of) any stratum of $K_d$
  defines a homeomorphism of the (closure of the) stratum onto
  its image which is (the closure of) a stratum of $\Pi _d$.

  (5) A stratum $S$ of $\Pi _d$ defined by a multiplicity vector with $\ell$
  components is a smooth $\ell$-dimensional real submanifold in $\mathbb{R}^d$.
  It is the graph of a smooth $(d-\ell )$-dimensional vector-function
  defined on the projection of the stratum in $Oa_{d-\ell}\ldots a_{d-1}$. Thus
  $S$ is a real manifold with boundary. The
  field of tangent spaces to $S$ continuously extends to the strata from the
  closure of~$S$. The extension is everywhere transversal to the space
  $Oa_0\ldots a_{d-\ell -1}$. That is, the sum of the two vector spaces
  $Oa_0\ldots a_{d-\ell -1}$ and (the extension of) the field of tangent spaces
  to $S$ is the space $Oa_0\ldots a_{d-1}$.

  (6) For $k\geq 3$, the set $\Pi _d^k$ is the set of points on and
  between the graphs
  $H^k_+$ and $H^k_-$ of two locally Lipschitz functions defined on
  $\Pi _d^{k-1}$ whose values coincide on and only on $\partial \Pi _d^{k-1}$:

  $$\begin{array}{ccl}
    \Pi _d^k&=&\{ (a_{d-k},a^{(d-k+1)})\in \mathbb{R}\times \Pi _d^{k-1}~
  |~H^k_-(a^{(d-k+1)})
  \leq a_{d-k}\leq H^k_+(a^{(d-k+1)})\} ~,\\ \\
  &&$$(H^k_-(a^{(d-k+1)})=H^k_+(a^{(d-k+1)}))\Leftrightarrow
  (a^{(d-k+1)}\in \partial \Pi _d^{k-1})~.\end{array}$$

  (7) For $k\geq 3$,
  the graph $H^k_+$ (resp. $H^k_-$) consists of the closures of the strata
  whose multiplicity vectors are of the form $(r,1,s,1,\ldots )$ (resp.
  $(1,r,1,s,\ldots )$) and which have exactly $k-1$ components. (In \cite{Ko}
  it is written ``$k$ components'' which is wrong.)

  (8) For $2\leq k\leq \ell$, the projection $S^k$ of every  
  $\ell$-dimensional stratum $S$ of $\Pi _d$ in the space
  $Oa_{d-k}\ldots a_{d-1}$
  is the set of points on and between the graphs $H^k_+(S)$
  and $H^k_-(S)$ of two locally Lipschitz functions defined on the closure
  $\overline{S^{k-1}}$ of $S^{k-1}$ whose
  values coincide on and only on $\partial S^{k-1}$.

\end{tm}

\begin{rems}\label{remsclosure}
  {\rm (1) The projections $\pi ^k$ are defined also for $k=2$.
    For $k=2$, each fibre $\tilde{f}_2$ is a half-line and only the graph
    $H_2^+$ (but not $H_2^-$) is defined,
    see Example~\ref{exk2}.

    (2) Consider two strata $S_1$ and $S_2$ of $\Pi _d$ defined by their
    multiplicity vectors $\mu (S_1)$ and $\mu (S_2)$. The stratum $S_2$
    belongs to the topological and algebraic closure
    of the stratum $S_1$ if and only if the vector $\mu (S_2)$ is obtained
    from the vector $\mu (S_1)$ by finitely-many replacings of two
    consecutive components by their sum.}
\end{rems}

\begin{rem}\label{remfibre}
  {\rm For $m\geq 2$, consider the fibres $f^{\diamond}_m$ of the projection}
  $$\pi ^m_*~:~\Pi _d\rightarrow \Pi _d^m~~~\, ,~~~\, 
  \pi ^m_*:=\pi ^{m+1}\circ \cdots \circ \pi ^d~.$$
      {\rm In particular, $\tilde{f}_d=f^{\diamond}_{d-1}$.
        Suppose that such a fibre $f^{\diamond}_m$ is over a point
        $A:=(a_{d-m}^0,\ldots ,a_{d-1}^0)\in \Pi _d^m$.
        When non-empty, the fibre $f^{\diamond}_m$ is either a point (when 
        $A\in \partial \Pi _d^m$) or a set homeomorphic to a
        $(d-m)$-dimensional cell
        and its boundary (when
        $A\in \Pi _d^m\setminus \partial \Pi _d^m$). This
        follows from part (6) of Theorem~\ref{tmremind}. 
        The boundary of the cell can be represented as consistsing
        of:

        -- two $0$-dimensional cells
        (these are the graphs
        of the functions $H^{m+1}_{\pm}|_A$),

        -- two $1$-dimensional cells
        (the graphs
        of $H^{m+2}_{\pm}|_{(\pi ^{m+1})^{-1}(A)}$),

        -- two $2$-dimensional cells
        (the graphs
        of $H^{m+3}_{\pm}|_{(\pi ^{m+1}\circ \pi ^{m+2})^{-1}(A)}$),

        -- $\ldots$,
        
        -- two
        $(d-m-1)$-dimensional cells
        (the graphs of
        $H^d_{\pm}|_{((\pi ^{m+1}\circ \pi ^{m+2}\circ \cdots \circ \pi ^{d-1})^{-1}(A)}$).}
  \end{rem}

\begin{rem}
  {\rm It is a priori clear that for the functions $L^k_{\pm}$
  defined in Theorem~\ref{tmconcrete}, one has the inequalities}

  $$L^k_+(a^{(d-k+1)})\leq H^k_+(a^{(d-k+1)})~~~{\rm and}~~~
    L^k_-(a^{(d-k+1)})\geq H^k_-(a^{(d-k+1)})$$
  {\rm for each value of $a^{(d-k-1)}$, where $L^k_+$ or $L^k_-$ (hence
    $H^k_+$ or $H^k_-$) is defined. It is also clear that the border of each
    component of the set $\Pi ^*_d$ consists of parts of the closures of
    the graphs $H^d_{\pm}$ and of parts of the hyperplanes
    $\{ a_j=0\}$, $j=1$, $\ldots$,~$d-1$.}
  \end{rem}

In Chapter~2 of \cite{Ko} one can find
also results concerning the hyperbolicity domain which are exposed in the
thesis \cite{Me} of I. M\'eguerditchian.

\section{Notation and examples\protect\label{secex}}

\begin{nota}\label{notaR}
  {\rm Given a $d$-tuple $\sigma =(\sigma _0,\ldots ,\sigma _{d-1})$, where
    $\sigma _j=+$ or~$-$, we denote
    by $\mathcal{R}(\sigma )$ the subset of
    $\mathbb{R}^d\cong Oa_0\cdots a_{d-1}$ defined by the conditions
    sign$(a_j)=\sigma _j$, $j=0$, $\ldots$, $d-1$, and we set
    $\Pi ^*_{d,\sigma}:=\Pi ^*_d\cap \mathcal{R}(\sigma )$. For a set
    $T\subset Oa_0\cdots a_{d-1}$, we denote by $T^k$ its projection in
    the space $Oa_{d-k}\cdots a_{d-1}$.}
  \end{nota}

\begin{ex}\label{exk1}
  {\rm For $k=1$ and for $a_j=0$, $j=0$, $\ldots$, $d-2$,
    there exists a hyperbolic polynomial of the form $(x+a_{d-1})x^{d-1}$
    with any $a_{d-1}\in \mathbb{R}$, so $\Pi _d^1=\mathbb{R}$. If one chooses
    any hyperbolic degree $d$ polynomial $Q_d^*$ with distinct roots, the
    shift $x\mapsto x+g$ results in $a_{d-1}\mapsto a_{d-1}+dg$, so there exist
    such polynomials $Q_d^*$ with any values of $a_{d-1}$. In addition, one
    can perturb the coefficients $a_0$, $\ldots$, $a_{d-2}$ to make them all
    non-zero by keeping the roots real and distinct. Thus
    $\Pi _d^{*1}=\mathbb{R}^*=\mathbb{R}\setminus \{ a_{d-1}=0\}$,}

  $$\Pi _d^{*1}\cap \{ a_{d-1}>0\} =\{ \mathbb{R}_+^*~:~a_{d-1}>0\} ~,~~~\,
  \Pi _d^{*1}\cap \{ a_{d-1}<0\} =\{ \mathbb{R}_-^*~:~a_{d-1}<0\} ~.$$
\end{ex}

\begin{ex}\label{exk2}
  {\rm One can formulate analogs to parts (1), (6) and (7) of
    Theorem~\ref{tmremind}
    for $k=2$ by saying that the border of the set
    $\Pi _d^2$ is the set $H^2_+$ while $H^2_-$ is empty, see part (1) of
    Remarks~\ref{remsclosure}. 
    
    The set $H^2_+$
    is the projection in $\mathbb{R}^2\cong Oa_{d-2}a_{d-1}$
    of the stratum of $\Pi _d$
    consisting of polynomials having a $d$-fold real root:
    $(x+\lambda )^d$. Its multiplicity vector equals $(d)$.
    Hence $a_{d-1}=d\lambda$, $a_{d-2}=d(d-1)\lambda ^2/2$, so
    $H^2_+:a_{d-2}=(d-1)a_{d-1}^2/2d$. 
One can observe
    that }
  $$\begin{array}{l}
  \Pi _d^{*2}=\{ a_{d-2}\neq 0\neq a_{d-1},~a_{d-2}<(d-1)a_{d-1}^2/2d\} ~,
  \\ \\
  \Pi _d^{*2}\cap \{ a_{d-1}>0,~a_{d-2}>0\} =
  \{ a_{d-1}>0,~0<a_{d-2}<(d-1)a_{d-1}^2/2d\} =:\Pi ^{*2}_{d++}~,\\ \\ 
  \Pi _d^{*2}\cap \{ a_{d-1}<0,~a_{d-2}>0\} =
  \{ a_{d-1}<0,~0<a_{d-2}<(d-1)a_{d-1}^2/2d\} =:\Pi ^{*2}_{d-+}~,\\ \\
  \Pi _d^{*2}\cap \{ a_{d-1}>0,~a_{d-2}<0\} =
  \{ a_{d-1}>0,~a_{d-2}<0\} =:\Pi ^{*2}_{d+-}~~~\, {\rm and}\\ \\ 
  \Pi _d^{*2}\cap \{ a_{d-1}<0,~a_{d-2}<0\} =
  \{ a_{d-1}<0,~a_{d-2}<0\} =:\Pi ^{*2}_{d--}~.\end{array}$$
{\rm To obtain similar formulas for $\Pi _d^2$ instead of $\Pi _d^{*2}$ one has
  to replace everywhere the inequalities $a_{d-1}<0$, $a_{d-1}>0$,
  $a_{d-2}<0$, $a_{d-2}>0$ and $a_{d-2}<(d-1)a_{d-1}^2/2d$ by $a_{d-1}\leq 0$,
  $a_{d-1}\geq 0$,
  $a_{d-2}\leq 0$, $a_{d-2}\geq 0$ and $a_{d-2}\leq (d-1)a_{d-1}^2/2d$ respectively.}
\end{ex}

\begin{figure}
\centerline{\hbox{\includegraphics[scale=0.4]{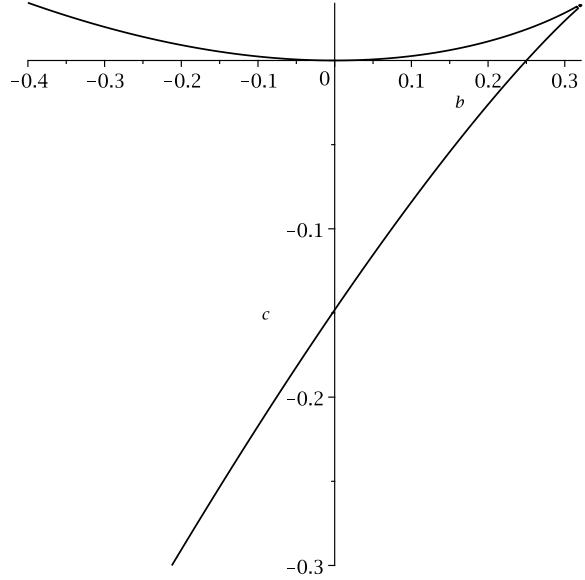}}}
\caption{The discriminant set of the family of polynomials $x^3+x^2+bx+c$
  and the sets $\Pi ^*_{3,\sigma}\cap \{ a=1\}$.}
\label{connect3fig}
\end{figure}

\begin{ex}\label{exk3}
  {\rm For $d=3$ (hence $\sigma =(\sigma _0,\sigma _1,\sigma _2)$),
    we set $a_2:=a$, $a_1:=b$, $a_0:=c$, and we consider the
    polynomial $Q_3:=x^3+ax^2+bx+c$. Taking into account the group of
    quasi-homogeneous dilatations which preserves the discriminant set
    (see part (2) of Remarks~\ref{remsdiscrim}) one concludes that each
  set $\Pi ^*_{3,\sigma}$ is diffeomorphic to the corresponding direct product

  $$(\Pi ^*_{3,\sigma}\cap \{ a=1\} )\times (0,\infty )~~~{\rm if}~~~\sigma _2=+
  ~~~{\rm or}~~~
  (\Pi ^*_{3,\sigma}\cap \{ a=-1\} )\times (-\infty ,0)~~~{\rm if}~~~\beta _2=-~.$$
  Set $\sigma ':=(-\sigma _0, \sigma _1,-\sigma _2)$. Using
  the same group of dilatations with $u=-1$ one deduces that the set
  $\Pi ^*_{3,\sigma '}\cap \{ a=-1\}$ is diffeomorphic to the set
  $\Pi ^*_{3,\sigma}\cap \{ a=1\}$. Therefore in order to prove that all sets
  $\Pi ^*_{3,\sigma}$ are contractible it suffices to show this for the sets
  $\Pi ^*_{3,\sigma}\cap \{ a=1\}$ with $\sigma _2=+$. The latter sets are shown
  in Fig.~\ref{connect3fig}. The figure represents the
  discriminant set of the polynomial $Q_3^{\bullet}:=x^3+x^2+bx+c$, i.~e. the set}

  $${\rm Res}(Q_3^{\bullet},{Q_3^{\bullet}}',x)=4b^3-b^2-18bc+27c^2+4c=0~.$$
  {\rm (The set $\Delta _3^2$ is empty, because there is not more than one 
    complex conjugate
  pair of roots, so $\Delta _3=\Delta _3^1$, see Remarks~\ref{remsdiscrim}.)  
    This is a curve in $\mathbb{R}^2:=Obc$ having a cusp at
      $(b,c)=(1/3,1/27)$ which corresponds to the polynomial $(x+1/3)^3$.
      The four sets $\Pi ^*_{3,\sigma}\cap \{ a=1\}$
      are the intersections of the interior of the curve with the open
      coordinate quadrants. The intersections with $\{ b>0,c>0\}$ and
      $\{ b>0,c<0\}$ are bounded curvilinear triangles.} 
  \end{ex}

\section{Proof of Theorem~\protect\ref{tm2roots}\protect\label{secprtm2roots}}

Part (1). Each set $\mathcal{E}_d(\sigma )$ is non-empty. Indeed, given a
polynomial $Q_d$ with $a_0>0$ (see Remark~\ref{remelliptic}),
for $C>0$ large enough, the polynomial $Q_d+C$
is elliptic. If the polynomials $Q_{d,1}$ and $Q_{d,2}$ belong
  to the set $\mathcal{E}_d(\sigma )$, then for
  $t\in [0,1]$, the polynomial $Q^{\sharp}_d:=tQ_{d,1}+(1-t)Q_{d,2}$
  also belongs to it.
  Indeed, the signs of the respective coefficients are the same and if
  $Q_{d,1}(x)>0$ and $Q_{d,2}(x)>0$, then $Q^{\sharp}_d(x)>0$. 
  Thus the set $\mathcal{E}_d(\sigma )$ is convex hence contractible.
  \vspace{1mm}
  
  Part (2). 
  Each set $\mathcal{F}_d(\sigma )$ is non-empty. Indeed, for $C>0$ large
  enough, the polynomial $Q_d+{\rm sign}(a_0)C$ has a single real root
  which is simple and the   
  sign of this root is opposite to the sign of $Q_d(0)$. For a given
  polynomial $Q_d\in \mathcal{F}_d(\sigma )$, denote this root by~$\xi$.
  Hence the polynomial $Q^0_d:=|\xi |^dQ_d(x/|\xi |)$
  is in $\mathcal{F}_d(\sigma )$
  and has a root at~$1$ or~$-1$. Suppose that the root is at~$1$ (for~$-1$
  the proof is similar). 
  We show that the subset $\mathcal{F}^0_d(\sigma )$ of
  $\mathcal{F}_d(\sigma )$ consisting of such polynomials $Q_d^0$
  is convex hence contractible. On the other hand the set
  $\mathcal{F}_d(\sigma )$ is diffeomorphic to
  $\mathcal{F}^0_d(\sigma )\times \mathbb{R}^*_+$ from which
  contractibility of $\mathcal{F}_d(\sigma )$ follows.

  For any two polynomials $Q_d^{0,\dagger}$,
  $Q_d^{0,*}\in \mathcal{F}^0_d(\sigma )$, the signs of the coefficients of the
  polynomial

  $$Q^{0,\flat}_d:=tQ_d^{0,\dagger}+(1-t)Q_d^{0,*},~t\in [0,1],$$
  are the same as the signs of the respective coefficients of
  $Q_d^{0,\dagger}$ and $Q_d^{0,*}$, so
  $Q^{0,\flat}_d\in \mathcal{F}^0_d(\sigma )$.
  This proves that $\mathcal{F}^0_d(\sigma )$ is convex. 
  \vspace{1mm}
  
  Part (3).
  \vspace{1mm}
  
  A) {\em Contractibility of the sets} $\mathcal{G}_{d,(2,0)}(\sigma )$ {\em and}
  $\mathcal{G}_{d,(0,2)}(\sigma )$.
  \vspace{1mm}
  
  The two real roots of $Q_d$ have the same sign (i.~e. $a_0>0$). 
  We assume that they are positive, i.~e. we prove contractibility only of
  $\mathcal{G}_{d,(2,0)}(\sigma )$; otherwise
  one can consider the polynomial $Q_d(-x)$ with the $d$-tuple
  $\tilde{\sigma}$ resulting
  from $\sigma$ via $x\mapsto -x$ (this mapping induces a bijection of the
  set of $d$-tuples onto itself) and contractibility of
  $\mathcal{G}_{d,(0,2)}(\tilde{\sigma})$ will be proved in the same way.
  Denote the real roots of $Q_d$ by $0<\xi <\eta$.
  
  We can assume that at least one coefficient of odd degree of $Q_d$
  is negative. Indeed, if all coefficients of $Q^0_d$
  of odd degree are positive,
  then by Descartes' rule
  of signs the polynomial $Q^0_d$ can have two real positive roots only if
  there is at least one coefficient of even degree which is negative.
  However in this case (and this is Case 1)) the set
  $\mathcal{G}_{d,(2,0)}(\sigma )$ is empty, see Proposition~4 in~\cite{FoKoSh}.

  Next, we assume that $\eta =1$ (hence $\xi \in (0,1)$).
  Indeed, if one considers instead of
  $Q_d\in \mathcal{G}_{d,(2,0)}(\sigma )$ the polynomial
  $Q^0_d:=\eta ^dQ_d(x/\eta )$, one has $Q^0_d\in \mathcal{G}_{d,(2,0)}(\sigma )$
  and $Q^0_d(1)=0$. We denote the set of such polynomials $Q^0_d$ by
  $\mathcal{G}^0_{d,(2,0)}(\sigma )$. As $\mathcal{G}_{d,(2,0)}(\sigma )$ is
  diffeomorphic to $\mathcal{G}^0_{d,(2,0)}(\sigma )\times \mathbb{R}^*_+$,
  contractibility of $\mathcal{G}^0_{d,(2,0)}(\sigma )$ implies the one
  of $\mathcal{G}_{d,(2,0)}(\sigma )$.  

  For $\xi ^*\in (0,1)$, we denote by $\mathcal{G}^{0,\xi ^*}_{d,(2,0)}(\sigma )$ the
  subset of polynomials of $\mathcal{G}^0_{d,(2,0)}(\sigma )$ with $\xi =\xi^*$.
  If $Q^{0,1}_d$ and $Q^{0,2}_d$ are two polynomials of
  $\mathcal{G}^{0,\xi ^*}_{d,(2,0)}(\sigma )$, then for $t\in [0,1]$, one has
  $tQ^{0,1}_d+(1-t)Q^{0,2}_d\in \mathcal{G}^{0,\xi ^*}_{d,(2,0)}(\sigma )$. Therefore
  for each $\xi \in (0,1)$, the set $\mathcal{G}^{0,\xi}_{d,(2,0)}(\sigma )$
  is convex hence contractible, and
  to prove contractibility of $\mathcal{G}^0_{d,(2,0)}(\sigma )$
  (and hence of $\mathcal{G}_{d,(2,0)}(\sigma )$) it suffices to find for
  each $\xi \in (0,1)$ a polynomial
  $Q^{0,\xi}_d\in \mathcal{G}^{0,\xi }_{d,(2,0)}(\sigma )$
  depending continuously on $\xi$. 

  Suppose that $m$ is odd, $1\leq m\leq d-1$, and that the coefficient of
  $Q_d\in \mathcal{G}_{d,(2,0)}(\sigma )$ of $x^m$ must be 
  negative. There exists a unique polynomial of the form

  \begin{equation}\label{eqdefiR}R:=x^d-Ax^m+B~,~~~A>0~,~~~
    B>0~,~~~{\rm such~that}~~~R(\xi )=R(1)=0~.
  \end{equation}
  Indeed, the conditions

  \begin{equation}\label{eqtwocond}
    \xi ^d-A\xi ^m+B=1-A+B=0
    \end{equation}
  imply

  \begin{equation}\label{eqABxi}
    A=(1-\xi ^d)/(1-\xi ^m)>0~~~{\rm and}~~~
    B=-1+A=\xi ^m(1-\xi ^{d-m})/(1-\xi ^m)>0~.
  \end{equation}

  \begin{rems}\label{remsAB}
    {\rm (1) The fractions for $A$, $B$ and $B/\xi ^m$ can be extended
      by continuity
      for $\xi =0$ and $\xi =1$. For $\xi \in [0,1]$, one has}

  \begin{equation}\label{eqlimit}
    \begin{array}{lll}A\in [1,\frac{d}{m}]~,&\lim _{\xi \rightarrow 0^+}A=1~,&
      \lim _{\xi \rightarrow 1^-}A=\frac{d}{m}~,\\ \\ 
      B\in [0,\frac{d-m}{m}]~,&\lim _{\xi \rightarrow 0^+}B=0^+~,&
      \lim _{\xi \rightarrow 1^-}B=\frac{d-m}{m}~,\\ \\ 
      B/\xi ^m\in [0,{\rm max}(\frac{d-m}{m},1)]~,&
      \lim _{\xi \rightarrow 0^+}B/\xi ^m=1~,&
      \lim _{\xi \rightarrow 1^-}B/\xi ^m=\frac{d-m}{m}~.
\end{array}\end{equation}

  {\rm (2) The function $R$ has a global minimum at some point
    $x_M=x_M(\xi )\in (0,1)$. One has}

  $$\lim _{\xi \rightarrow 0^+}x_M(\xi )=x_{M,0}=(m/d)^{1/(d-m)}\in (0,1)~,~~~
  R(x_{M,0})<0~~~{\rm and}~~~
    \lim _{\xi \rightarrow 1^-}x_M(\xi )=1~.$$
         {\rm For $m\geq 3$, the tangent line to the graph of $R$ for $x=0$
           is horizontal and $(0,R(0))$ is an inflection point. There is
           also another inflection point $x_I=x_I(\xi )\in (0,x_M)$.}
  \end{rems}

  Set

  \begin{equation}\label{eqdefiI}
    \mathcal{I}:=\{ 1,2,,\ldots ,m-1,m+1,m+2,\ldots ,d-1\} ~.
    \end{equation}
  We construct a polynomial $\Psi :=\sum _{j=0}^{d-1}\psi _jx^j$ with signs of
  the coefficients $\psi _j$, $j\in \mathcal{I}$, as defined
  by the $d$-tuple $\sigma$ and satisfying the conditions

  \begin{equation}\label{eqPsi}\Psi (\xi )=\Psi (1)=0~.
  \end{equation}
  The latter conditions can be considered as a linear system with
  unknown variables $\psi _0$ and $\psi _m$. Its determinant equals
  $\xi ^m-1\neq 0$, so for given $\psi _j$, $j\in \mathcal{I}$, 
  these conditions define a unique couple $(\psi _0, \psi _m)$ whose
  signs are not necessarily the ones defined by the $d$-tuple $\sigma$.
  So to construct $\Psi$ it suffices to fix $\psi _j$ for $j\in \mathcal{I}$.

  For each $\xi \in (0,1)$ fixed and for $\varepsilon >0$
  sufficiently small, one has
  $R+\varepsilon \Psi \in \mathcal{G}^{0,\xi}_{d,(2,0)}(\sigma )$. Indeed,
  for $m\neq j\neq 0$, the coefficients of $R+\varepsilon \Psi$ have the
  signs defined by the $d$-tuple $\sigma$, so one has to check two things:
  \vspace{1mm}
  
  1) If
  $\varepsilon$ is small enough, then

  \begin{equation}\label{eqAB}-A+\varepsilon \psi _m<0~~~{\rm and}~~~
    B+\varepsilon \psi _0>0~.
  \end{equation}
  To obtain these two conditions simultaneously
  for all $\xi \in (0,1)$, one has to choose $\varepsilon$ as a function
  of $\xi$.

  The conditions (\ref{eqPsi}) can be given the form

  $$\xi ^m\psi _m+\psi _0=U~,~\psi _m+\psi _0=V~,$$
  where $U$ and $V$ are polynomials in $\xi$ of degree $\leq d-1$. Hence

  \begin{equation}
    \label{eqpsi0m}
    \psi _0=(U-\xi ^mV)/(1-\xi ^m)~~~{\rm and}~~~\psi _m=(V-U)/(1-\xi ^m)~.
    \end{equation}
  Formulas (\ref{eqpsi0m}) imply that $\Psi$ is of the form

  \begin{equation}\label{eqKxxi}
    K(x,\xi )/(1-\xi )^m~,~~~
    K\in \mathbb{R}[x,\xi ]~,~~~{\rm deg}_xK\leq d-1~.
    \end{equation}
  As $\xi \rightarrow 0^+$, the quantity $B$ decreases as $\xi ^m$, see
  (\ref{eqABxi}) and (\ref{eqlimit}). As
  $\xi \rightarrow 1^-$, the quantities $|\psi _0|$ and $|\psi _m|$
  increase not faster than $C/(1-\xi )$ for some $C>0$. So to obtain
  $\varepsilon =\varepsilon (\xi )$ such that conditions (\ref{eqAB})  hold
  for $\xi \in (0,1)$, it suffices to set $\varepsilon :=c\xi ^{m+1}(1-\xi )^3$
  for some $c>0$ small enough.
  \vspace{1mm}

  2) For $\xi \in (0,1)$, one must have

  \begin{equation}\label{eqRpsi}
    R+\varepsilon \Psi >0~~~{\rm for}~~~
  x\in (-\infty ,\xi )\cup (1,\infty ),~~~{\rm and}~~~R+\varepsilon \Psi <0
  ~~~{\rm for}~~~x\in (\xi ,1)~.
  \end{equation}

  \begin{lm}\label{lmRpsi}
    It is possible to choose $c>0$ so small
    that conditions (\ref{eqAB}) and (\ref{eqRpsi}) hold true simultaneously.
  \end{lm}

  The lemma implies that for such $c>0$, 
  $R+\varepsilon (\xi )\Psi \in \mathcal{G}^{0,\xi}_{d,(2,0)}(\sigma )$. So one
  can set $Q^{0,\xi}_d:=R+\varepsilon (\xi )\Psi$ from which contractibility of
  $\mathcal{G}_{d,(2,0)}(\sigma )$ follows. 

  \begin{proof}[Proof of Lemma~\ref{lmRpsi}]
    Conditions (\ref{eqAB}) were already discussed, so we focus on conditions
    (\ref{eqRpsi}). Lowercase indices $\xi$ indicate differentiations
    w.r.t.~$\xi$. 
    \vspace{1mm}
    
    a) To obtain the condition $R+\varepsilon \Psi >0$
  for $x>1$, it suffices to get $(R+\varepsilon \Psi )'>0$ for $x\geq 1$.
  For $x\geq 1$, one has

  \begin{equation}\label{eqRprime}
    R'=dx^{d-1}-mAx^{m-1}=x^{m-1}(dx^{d-m}-mA)\geq x^{m-1}(d-mA)
    \end{equation}
  (as $R'(1)>0$, one knows that $d-mA>0$). Next, 

  $$d-mA=\Lambda /(1-\xi ^m)~,~~~\Lambda :=d-m+m\xi ^d-d\xi ^m~.$$
  There exists $\alpha >0$ such that for $\xi \in [0,1]$,
  $\Lambda \geq \alpha (1-\xi )^2$. Indeed,
  $\Lambda _{\xi}=dm(\xi ^{d-1}-\xi ^{m-1})\leq 0$, with equality only for
  $\xi =0$ and $\xi =1$, so $\Lambda$ is strictly decreasing on $[0,1]$.
  The existence of $\alpha$ follows from

  $$\begin{array}{ll}
    \Lambda (0)=d-m>0~,~~~\Lambda (1)=\Lambda _{\xi}(1)=0&{\rm and}\\ \\
    \Lambda _{\xi \xi}=dm((d-1)\xi ^{d-1}-(m-1)\xi ^{m-1})&{\rm hence}\\ \\
    \Lambda _{\xi \xi}(1)=dm(d-m)>0~.&\end{array}$$
  Thus for $\xi \in (0,1)$ and $x>1$, one has

  $$R\geq (x^m/m)\alpha (1-\xi )^2/(1-\xi ^m)~~~\, {\rm and}~~~\, 
  \varepsilon (\xi )\Psi \leq c\xi ^{m+1}(1-\xi )^3K(x,\xi )/(1-\xi ^m)~,$$
  see (\ref{eqKxxi}). 
  One can choose $c>0$
  sufficiently small so that for $x\in (1,2]$, $R+\varepsilon \Psi >0$. There
    exists $\beta >0$ such that for $x\geq 2$, $dx^{d-m}-mA>\beta x^{d-m}$
    (see (\ref{eqRprime}) and (\ref{eqlimit})),
    so $R\geq \beta x^d/d$ and for $c>0$
    small enough, $R+\varepsilon \Psi >0$.
    \vspace{1mm}
    
    b) For $x\leq -1$ (resp. for $x\in [-1,0]$),
    one has

    $$R\geq |x|^m(|x|^{d-m}+A)~~~\, {\rm (resp.}~~~\, 
    R\geq B\geq ({\rm max}((d-m)/m,1))\xi ^m{\rm )}$$
    (see (\ref{eqdefiR}) and (\ref{eqlimit})) which
    for $c>0$ small enough is larger than $|\varepsilon (\xi )\Psi |$ and
    (\ref{eqRpsi}) holds true.
    \vspace{1mm}

    c) Suppose that $x\in (0,\xi )$. Then $R\geq \min (h(x,\xi ), q(x,\xi ))$,
    where

    $$\tau ~:~y=h(x,\xi):=R'(\xi )(x-\xi )~~~\, {\rm and}~~~\, \chi ~:~
    y=q(x,\xi ):=B-Bx/\xi$$
    are the tangent line to the graph of $R$ at the point $(\xi ,0)$ and
    the line joining the points $(0,B)$ and $(\xi ,0)$ respectively. Indeed,
    if $x_I\in [\xi ,1)$ (see part (2) of Remarks~\ref{remsAB}),
      then the graph of $R$ is concave for $x\in [0,\xi]$,
      so it is situated above the line $\chi$. If $x_I\in (0,\xi )$, then for
      $x\in [x_I,\xi ]$, one has $R\geq h(x,\xi )$ and for $x\in (0,x_I]$,
    one has $R\geq q_1(x,\xi )$, where

    $$\chi _1~:~y=q_1(x,\xi ):=R(x_I)+(x-x_I)(R(x_I)-B)/x_I$$
    is the line joining the points $(0,B)$ and $(x_I,R(x_I))$.
    The line $\chi _1$ is above the line $\chi$ for $x\in (0,x_I)$.

    Consider the smaller in absolute value of the slopes of the lines $\tau$
    and $\chi$, i.e. $\mu :=\min (|R'(\xi )|,B/\xi )$. One finds that

    $$R'(\xi )=\xi ^{m-1}g(\xi )/(1-\xi ^m)~,~~~g:=d\xi ^{d-m}-m-(d-m)\xi ^d~,$$
    with $g_{\xi}=d(d-m)(\xi ^{d-m-1}-\xi ^{d-1})\geq 0$, with equality only for
    $\xi =0$ and $\xi =1$. As $g(0)=-m<0$, $g(1)=0$,

    $$g_{\xi \xi}=d(d-m)((d-m-1)\xi ^{d-m-2}-(d-1)\xi ^{d-2})~,~~~\, {\rm so}~~~\,
    g_{\xi \xi}(1)=-md(d-m)<0~,$$
    there exists $\tilde{\beta}>0$ such that for $\xi \in (0,1)$,
    $|R'(\xi )|\geq \tilde{\beta}\xi ^{m-1}(1-\xi )^2/(1-\xi ^m)$.
    On the other hand
    $B/\xi =\xi ^{m-1}(1-\xi ^{d-m})/(1-\xi ^m)$. Thus
    $\mu \geq \mu _0:=\gamma \xi ^{m-1}(1-\xi )$ for some $\gamma >0$. Hence
    for $x\in (0,\xi )$, the graph of $R$ is above the line
    $\delta :y=-\mu _0(x-\xi )$.

    There exists $D_0>0$ such that for $\xi \in (0,1)$ and $x\in [0,1]$,
    one has $|(1-\xi )\Psi '|\leq D_0$, see (\ref{eqpsi0m}). Hence if $c>0$
    is sufficiently small,
    the graph of $\varepsilon \Psi$ is below the line $\delta$ for
    $x\in [0,\xi )$, so $R+\varepsilon \Psi >0$.
\vspace{1mm}

d) Suppose that $m\geq 3$ and that $\xi >0$ is close to $0$.
Then for $x>\xi$, the 
line $\tilde{\tau}$, which is tangent to the graph of $R$ at the point
$(\xi ,0)$, is above the straight line $\tilde{\rho}$ joining the points
$(\xi ,0)$ and $(x_M,R(x_M))$. Indeed,

$$R'(\xi )=\xi ^{m-1}(d\xi ^{d-m}-m-(d-m)\xi ^d)/(1-\xi ^m)=O(\xi ^{m-1})$$
whereas the slope of $\tilde{\rho}$ is close to $R(x_{M,0})/x_{M,0}<0$.
Therefore for
$x\in (\xi ,x_M]$, the graph of $R$ is below the line $\tilde{\tau}$.

    For
    $x\in [x_M,1)$, the graph of $R$ is below the line $\tilde{\chi}$
      joining the points $(x_M,R(x_M))$ and $(1,0)$ whose slope
      $-R(x_M)/(1-x_M)$ is close to $-R(x_{M,0})/(1-x_{M,0})>0$. On the other hand
      one has $|(1-\xi )\Psi '|\leq D_0$ (see c)), so
      $|\varepsilon (\xi )\Psi '|\leq c\xi ^{m+1}(1-\xi )^2D_0$. Thus
      the graph of $\varepsilon (\xi )\Psi$ is above the line $\tilde{\tau}$
      for $x\in (\xi ,x_M]$ and above $\tilde{\chi}$ for $x\in [x_M,1)$, hence
      it is between the graph of $R$ and the $x$-axis for $x\in (\xi ,1)$,
      so $R+\varepsilon \Psi <0$. 
      \vspace{1mm}

      e) For $m\geq 3$, we fix $\theta _0>0$ small enough such that
      for $\xi \in (0,\theta _0]$, $R+\varepsilon \Psi <0$, see d). 
      For $m\geq 3$,
      $\xi \in [\theta _0,1]$, $x\in (\xi ,1)$, and for $m=1$,
      $\xi \in [0,1]$, $x\in (\xi ,1)$, one has $R+\varepsilon (\xi )\Psi <0$
      if $c>0$ is small enough. Indeed, one can write

      $$R=(x-1)(x-\xi )R_1~~~\, {\rm and}~~~\, \Psi =(x-1)(x-\xi )\Psi _1~,~~~
      R_1,~\Psi _1\in \mathbb{R}[x,\xi]~.$$
      Then $R_1(x,\xi )>0$. In particular, for $\xi =1$, one obtains

      $$R=x^d-(d/m)x^m+(d-m)/m~,~~~\, R'=dx^{d-1}-dx^{m-1}~,~~~\, {\rm so}~~~\,
      R'(1)=0~,$$
      and $R''=d((d-1)x^{d-2}-(m-1)x^{m-2})$ hence $R''(1)=d(d-m)>0$, i.~e. $R$
      is divisible by $(x-1)^2$, but not by $(x-1)^3$.

      For $m=1$, $\xi =0$,
      one has $R'(0)<0$ (whereas for $m=3$, $\xi =0$, one has $R'(0)=0$),
      this why for $m=1$ our reasoning is valid for $\xi \in [0,1]$, not
      only for $\xi \in [\theta _0,1]$.

      Denote by $R_{1,0}>0$ the minimal value of $R_1$ and by $\Psi _{1,0}$ the
      maximal value of $\Psi _1$ for $x\in [0,1]$. One can choose $c>0$ so small
      that for $x\in (\xi ,1)$ and for the values of $\xi$ mentioned at the
      beginning of e), 

      $$R_1-\varepsilon \Psi _1\geq R_{1,0}-\varepsilon \Psi _{1,0}>0~,~~~\,
      {\rm so}~~~\, R+\varepsilon \Psi <0~,~~~\,
      {\rm because}~~~\, (x-1)(x-\xi )<0~.$$
      The proof of the lemma results from a) -- e).
  \end{proof}
  \vspace{1mm}
  
  B) {\em Contractibility of the set} $\mathcal{G}_{d,(1,1)}(\sigma )$.
  \vspace{1mm}
  
  The two real roots of $Q_d$ have opposite signs
  (hence $a_0<0$). Denote them by $-\eta <0<\xi$. We define the sets 

$$\mathcal{K}:=\mathcal{G}_{d,(1,1)}(\sigma )\cap \{ \xi >\eta \}~,~~~
  \mathcal{L}:=\mathcal{G}_{d,(1,1)}(\sigma )\cap \{ \xi <\eta \} ~~~{\rm and}~~~
  \mathcal{M}:=\mathcal{G}_{d,(1,1)}(\sigma )\cap \{ \xi =\eta \} ~.$$

  \begin{lm}\label{lmKLM}
    Set $\sigma :=(\sigma _0,\ldots ,\sigma _{d-1})$, $\sigma _j=+$ or~$-$.
    
    (1) Suppose that $\sigma _{2j+1}=+$, $j=0$, $1$, $\ldots$, $(d/2)-1$.
    Then $\mathcal{K}=\mathcal{M}=\emptyset$.

    (2) Suppose that $\sigma _{2j+1}=-$, $j=0$, $1$, $\ldots$, $(d/2)-1$.
    Then $\mathcal{L}=\mathcal{M}=\emptyset$.

    (3) Suppose that there exist two odd integers
    $j_1\neq j_2$, $1\leq j_1,j_2\leq d-1$, such that
    $\sigma _{j_1}=-\sigma _{j_2}$. Then all three sets
    $\mathcal{K}$, $\mathcal{L}$ and $\mathcal{M}$ are non-empty. There exists
    an open $d$-dimensional
    ball $\mathcal{B}\subset \mathcal{G}_{d,(1,1)}(\sigma )$
    centered at a point in $\mathcal{M}$ and such that  
    $\mathcal{B}\cap \mathcal{K}\neq \emptyset$ and
    $\mathcal{B}\cap \mathcal{L}\neq \emptyset$.
  \end{lm}

  \begin{proof}
    Parts (1) and (2). If
    $\sigma _{2j+1}=+$ (resp. $\sigma _{2j+1}=-$),
    $j=0$, $1$, $\ldots$, $(d/2)-1$, then
    for a polynomial $Q_d\in \mathcal{G}_{d,(1,1)}(\sigma )$, one has
    $Q_d(0)<0$ and $Q_d(a)>Q_d(-a)$ (resp. $Q_d(0)<0$ and $Q_d(a)<Q_d(-a)$)
    for $a>0$. Hence $\xi <\eta$ (resp. $\xi >\eta$).

    Part (3). We construct a polynomial
    $Q_d^{\diamond}\in \mathcal{M}$. Set $u:=\xi ^{j_1-j_2}$ and 

    $$Q_d^{\diamond}:=x^d-\xi ^d+\sigma_{j_1}(x^{j_1}-ux^{j_2})+
    \varepsilon (Q_d^{\diamond ,o}+Q_d^{\diamond ,e})~,$$
    where 

    $$Q_d^{\diamond ,e}=b+\sum _{j=1}^{d/2}\sigma_{2j}x^{2j}~,~~~
    b\in \mathbb{R}~,~~~ 
    Q_d^{\diamond ,o}=rx^{j_1}+\sum _{j=0}^{d/2-1}\sigma_{2j+1}x^{2j+1}$$
    and
    $\varepsilon >0$ is small enough. We choose $b$ and $r$ such that
    $Q_d^{\diamond ,e}(\pm \xi )=0$ and $Q_d^{\diamond ,o}(\pm \xi )=0$
    respectively. Then $Q_d^{\diamond}(\pm \xi )=0$ and for $j\neq 0$ and 
    $j_1\neq j\neq j_2$, the sign of
    the coefficients of $x^j$ of $Q_d^{\diamond}$ is as defined by $\sigma$. For
    $\varepsilon >0$ small enough, one has
    sign$(Q_d^{\diamond}(0))$=sign$(-\xi ^d+\varepsilon b)=-$.
    The coefficient of $x^{j_1}$ (resp. $x^{j_2}$) of $Q_d^{\diamond}$ equals
    $\sigma_{j_1}\times (1+\varepsilon (1+r))$
    (resp. $\sigma_{j_2}\times (u+\varepsilon (1+r))$), so it has the same sign
    as~$\sigma _{j_1}$ (resp. as~$\sigma _{j_2}$).

    Consider a $d$-dimensional ball $\mathcal{B}$
    centered at a point $Q_d^{\diamond}\in \mathcal{M}$, with $\xi =\eta =\xi _0$
    and belonging
    to $\mathcal{G}_{d,(1,1)}(\sigma )$. Perturb the real root $\xi$ of
    $Q_d^{\diamond}$ so that it takes values smaller and values larger
    than~$\xi _0$. The signs of the coefficients of $Q_d^{\diamond}$ do not
    change. Hence $\mathcal{B}$ intersects $\mathcal{K}$ and~$\mathcal{L}$.
  \end{proof}
  We show first that each
  of the two sets $\mathcal{K}$ and $\mathcal{L}$, when nonempty,
  is contractible. If we are in the conditions of part (1) or~(2) of
  Lemma~\ref{lmKLM}, then this implies contractibility of
  $\mathcal{G}_{d,(1,1)}(\sigma )$. When we are in
  the conditions of part~(3), then one can contract
  $\mathcal{K}$ and $\mathcal{L}$ into points of $\mathcal{B}$
  and then contract $\mathcal{B}$ into a point,
  so in this case $\mathcal{G}_{d,(1,1)}(\sigma )$ is also contractible.

  We prove contractibility only of $\mathcal{K}$ (when non-empty). The one
  of $\mathcal{L}$ is performed by complete analogy (the change of variable
  $x\mapsto -x$ exchanges the roles of $\mathcal{K}$ and $\mathcal{L}$ and
  changes the $d$-tuple $\sigma$ accordingly).
  So we suppose that $\xi >\eta$. As in the proof of A) we reduce the proof
  of the contractibility of $\mathcal{K}$ to the one of the contractibility
  of $\mathcal{K}\cap \{ \xi =1\}$. As in A) we observe that
  if

  $$Q^{\ddagger}_d~,~~~Q^{\triangle}_d\in \mathcal{K}^{\eta^*}:=
  \mathcal{K}\cap \{ \xi =1, \eta =\eta ^*\in (0,1)\} ~,$$
  then 
  $tQ^{\ddagger}_d+(1-t)Q^{\triangle}_d\in \mathcal{K}^{\eta^*}$, 
  so $ \mathcal{K}^{\eta^*}$ is convex hence contractible and contractibility of
  $\mathcal{K}\cap \{ \xi =1\}$ (and also of $\mathcal{K}$) will be proved if
  we construct for each $\eta \in (0,1)$ a polynomial
  $Q_d\in \mathcal{K}^{\eta}$ depending continuously on~$\eta$.

  Suppose that there is a negative coefficient of $Q_d$ of odd degree~$m$
  (otherwise $\mathcal{K}$ is empty).
  For $\eta \in (0,1)$, we construct a polynomial

  $$S:=x^d-\tilde{A}x^m-\tilde{B}~,~~~\tilde{A}>0~,~~~\tilde{B}>0~,~~~
  {\rm such~that}~~~
  S(1)=S(-\eta )=0~.$$
  The latter two equalities imply

  \begin{equation}\label{eqABbis}
    \tilde{A}=(1-\eta ^d)/(1+\eta ^m)>0~{\rm and}~
    \tilde{B}=\eta ^m(1+\eta ^{d-m})/(1+\eta ^m)>0~.
    \end{equation}
  \begin{rems}\label{remsS}

    {\rm (1) Thus for $\eta \in [0,1]$, there exist constants
      $0<B_{{\rm min}}\leq B_{{\rm max}}$ such that
      $\tilde{B}/\eta ^m\in [B_{{\rm min}},B_{{\rm max}}]$. Moreover one has} 

  \begin{equation}\label{eqlimbis}
    \begin{array}{lll}\tilde{A}\in [0,1]~,&\lim _{\eta \rightarrow 0^+}\tilde{A}=1~,&
      \lim _{\eta \rightarrow 1^-}\tilde{A}=0^+~,\\ \\ \tilde{B}\in [0,1]~,&
      \lim _{\eta \rightarrow 0^+}\tilde{B}=0^+~,&
      \lim _{\eta \rightarrow 1^-}\tilde{B}=1~,\\ \\
      \lim _{\eta \rightarrow 0^+}\tilde{B}/\eta ^m=1
      &{\rm and}&\lim _{\eta \rightarrow 1^-}\tilde{B}/\eta ^m=1~.
    \end{array}
  \end{equation}

  {\rm (2) The derivative $S'$ has a unique root $\tilde{x}_M$
    (which is simple) in $(0,1)$.
    All non-constant derivatives of $S$ are increasing for $x>\tilde{x}_M$, 
    have one or two roots (depending on $m$) in $[0,\tilde{x}_M)$
    and no root outside
    this interval.}
  \end{rems}
  
  We construct a polynomial $\Phi :=\sum _{j=0}^{d-1}\varphi _jx^j$, where for
  $j\in \mathcal{I}$ (see (\ref{eqdefiI})),
  the sign of $\varphi _j$ is defined by the
  $d$-tuple~$\sigma$. This polynomial must satisfy the condition

  $$\Phi (-\eta )=\Phi (1)=0$$
  which can be regarded as a linear system with known quantities $\varphi _j$,
  $j\in \mathcal{I}$, and with unknown variables $\varphi _0$ and $\varphi _m$:

  \begin{equation}\label{eqWT}
    \begin{array}{lll}-\eta ^m\varphi _m+\varphi _0=W~,&
    \varphi _m+\varphi _0=T~,&
    W,T\in \mathbb{R}[\eta ]~,~~~{\rm so}\\ \\
    \varphi _0=(\eta ^mT+W)/(1+\eta ^m)~,&\varphi _m=(T-W)/(1+\eta ^m)~.&
    \end{array}
    \end{equation}
  One must also have $S+\varepsilon _1(\eta )\Phi\in \mathcal{K}^{\eta}$,
  $\eta \in (0,1)$,
  for some suitably chosen positive-valued continuous
  function $\varepsilon _1(\eta )$.
  For $\varepsilon _1(\eta )>0$ small enough, the sign of the coefficient of
  $x^j$, $j\in \mathcal{I}$, of the polynomial
  $S+\varepsilon _1(\eta )\Phi$ is as defined by
  the $d$-tuple~$\sigma$. So one needs to choose $\varepsilon _1(\eta )$
  such that

  \begin{equation}\label{eqsigns}
    \tilde{A}+\varepsilon _1(\eta )\varphi _m<0~,~-
    \tilde{B}+\varepsilon _1(\eta )\varphi _0<0
  \end{equation}
  and
  \begin{equation}\label{eqsignsbis}
    S+\varepsilon _1(\eta )\Phi >0~~~\, {\rm for}~~~\, 
      x\in (-\infty ,-\eta )\cup (1,\infty)~,~~~S+\varepsilon _1(\eta )\Phi <0
      ~~~\, {\rm for}~~~\, x\in (-\eta ,1)~.
  \end{equation}
  We set $\varepsilon _1:=\tilde{c}\eta ^m(1-\eta )^2$, $\tilde{c}>0$.
  If one chooses $\tilde{c}$ small
  enough, conditions (\ref{eqsigns}) will hold true.

  \begin{lm}\label{lmS}
    For $\tilde{c}>0$ small enough, conditions (\ref{eqsignsbis}) hold true.
  \end{lm}
  
  Contractibility of $\mathcal{K}$ follows from the lemma.
  
  \begin{proof}[Proof of Lemma~\ref{lmS}]
    All derivatives of $S$ of order $\leq d-1$ are increasing functions in $x$
    for $x\geq 1$ (see Remarks~\ref{remsS}). As

    $$S'(1)=(d+d\eta ^m-m+m\eta ^d)/(1+\eta ^m)\geq (d-m)/2~,$$
    one can choose $\tilde{c}$ small enough so that for $x\in [1,2]$,
    $S'+\varepsilon _1(\eta )\Phi '>0$. Hence $S+\varepsilon _1(\eta )\Phi >0$
    for $x\in (1,2]$. If $x\geq 2$, then for some positive constants $k_1$
    and $k_2$, one has $S'\geq k_1x^{d-1}$ and
    $\Phi '\leq k_2x^{d-2}$, so if
    $\tilde{c}>0$ is small enough, then for $x\geq 2$ (hence for $x>1$),
    $S'+\varepsilon _1(\eta )\Phi '>0$ and $S+\varepsilon _1(\eta )\Phi >0$.

    One has

    $$S'(-\eta )=-(d\eta ^{d-1}+(d-m)\eta ^{d+m-1}+m\eta ^{m-1})/(1+\eta ^m)=
    O(\eta ^{m-1})~,$$
    $S'(-\eta )<0$ and $S$ is convex for $x<0$. Hence
    one can choose $\tilde{c}>0$ so small that for $x\in [-2,-\eta ]$,
    $S'+\varepsilon _1(\eta )\Phi '<0$ hence $S+\varepsilon _1(\eta )\Phi >0$.
    Indeed, for $\eta \in [0,1]$ and $x\in [-2,0]$, $\Phi '$ is bounded. 
    For $x\leq -2$, one has

    $$S'\leq k_1^*x^{d-1}~~~\, \, {\rm and}~~~\, \, 
    |\Phi '|\leq k_2^*x^{d-2}$$
    for some positive constants $k_1^*$, $k_2^*$,
    so  $S+\varepsilon _1(\eta )\Phi <0$ 
    (thus this holds true for $x<-\eta$).

    The function $S$ is convex on $[-\eta ,0]$, see Remarks~\ref{remsS}.
    Hence for $x\in [-\eta ,0]$, the graph of $S$ is below the line $\zeta$
    joining
    the points $(-\eta ,0)$ and $(0,-\tilde{B})$. Its slope is
    $-\tilde{B}/\eta$, with
    $|-\tilde{B}/\eta |=O(\eta ^{m-1})$. Hence for $x\in [-\eta ,0]$ and for
    $\tilde{c}>0$ sufficiently small, the graph
    of $\Phi$ is above the line $\zeta$ (because $\Phi '$ is bounded for
    $x\in [-1,0]$, $\eta \in [0,1]$)
    and one has
    $S+\varepsilon _1(\eta )\Phi <0$.

    Suppose that $x\in [0,\tilde{x}_M]$. The function $S$ is decreasing, see
    Remarks~\ref{remsS}, hence $S(x)\leq S(0)=-\tilde{B}=O(\eta ^m)$.
    As there exists $k_3>0$ such that
    for $x\in [0,1]$, $|\Phi |\leq k_3$, for $\tilde{c}>0$
    sufficiently small,
    one has $S+\varepsilon _1(\eta )\Phi <0$.

    For $x\in [\tilde{x}_M,1]$, the function $S$ is convex.
    Hence its graph is below
    the line $\tilde{\zeta}$ joining the points $(\tilde{x}_M,S(\tilde{x}_M))$
    and $(1,0)$.
    Recall that $S(\tilde{x}_M)\leq S(0)=-\tilde{B}=O(\eta ^m)$.
    There exists $k_4>0$ such that
    for $x\in [0,1]$ and $\eta \in [0,1]$, $|\Phi '|\leq k_4$. 
    Thus the slope of $\tilde{\zeta}$
    is

    $$\geq \tilde{B}/(1-\tilde{x}_M)>\tilde{B}=O(\eta ^m)$$
    while
    $|\varepsilon \Phi '|\leq \tilde{c}\eta ^m(1-\eta )^2k_4$.
    Hence for sufficiently
    small values of $\tilde{c}>0$, the graph of $\varepsilon \Phi$
    is above the line
    $\tilde{\zeta}$ and $S+\varepsilon _1(\eta )\Phi <0$. 
    \end{proof}

\section{Proofs of
  Theorems~\protect\ref{tmconnecthyp} and~\protect\ref{tmconcrete}
  \protect\label{secprtmconnecthyp}}

\begin{proof}[Proof of Theorem~\ref{tmconnecthyp}]
  In the proof we assume that the polynomials of $\Pi _d$ are of the form
  $Q_{d}:=x^{d}+a_{d-1}x^{d-1}+\cdots +a_2x^2+a_1x+a_0$ and the ones of 
  $\Pi _{d-1}$ are of the form $Q_{d-1}:=x^{d-1}+a_{d-1}x^{d-2}+\cdots +a_2x+a_1$.
  Thus the intersection $\Pi _d\cap \{ a_0=0\}$ can be identified with
  $\Pi _{d-1}$.
  
  We show that every polynomial $Q_d\in \Pi _d^*$ can be
    continuously deformed so that it remains in $\Pi _d^*$, the signs of its
    coefficients do not change throughout the deformation except the one of
    $a_0$ which vanishes at the end of the deformation. Therefore
\vspace{1mm}

1) throughout the deformation the quantities of positive and negative
roots do not change;
\vspace{1mm}

2) at the end of
the deformation exactly one root vanishes and
a polynomial of the form $xQ_{d-1}$ is obtained with 
$Q_{d-1}\in \Pi _d\cap \{ a_0=0\}$.
\vspace{1mm}

Moreover, we show that throughout and at the end of
    the deformation
    one obtains polynomials with distinct real roots. Thus any given component
    of the set $\Pi _d^*$ can be retracted into a component of the set
    $\Pi _{d-1}^*$; the latter is defined by the $(d-1)$-tuple obtained from
    $\sigma$ by deleting its first component.
    For $d=2$, all components of the set $\Pi _2^*$ are contractible, see
    Example~\ref{exk2}.

    This means that for
    every given $d$ and $\sigma$, there exists exactly one component of
    $\Pi _d^*$, and 
    which is contractible. The deformation mentioned above is defined like this:

    $$Y_d:=(Q_d+txQ_d')/(1+td)=\sum _{j=0}^d((1+tj)/(1+td))a_jx^j~,~~~t\geq 0~.$$
    It is clear that the polynomial $Y_d$ is monic, with
    sign$(a_j)=$sign$((1+tj)a_j/(1+td))$ and
    $\lim _{t\rightarrow +\infty}((1+tj)a_j/(1+td))=ja_j/d$.
    There remains to prove only that $Y_d$ has $d$ distinct real roots.

    Denote the roots of $Q_d$ by
    $\eta _1<\cdots <\eta _s<0<\xi _1<\cdots <\xi _{d-s}$. The polynomial $Q_d'$
    has exactly one root in each of the intervals $(\eta _1,\eta _2)$,
    $\ldots$,
    $(\eta _{s-1},\eta _s)$, $(\eta _s,\xi _1)$,
    $(\xi _1,\xi _2)$, $\ldots$,
    $(\xi _{d-s-1},\xi _{d-s})$. We denote these roots by
    $\tau _1<\cdots <\tau _{d-1}$.

    For each $t\geq 0$,
    the polynomial $Y_d$ changes sign in each of the intervals
    $(\eta _j,\tau _j)$, $j=1$, $\ldots$, $s-1$, and in each of the
    intervals $(\tau _{s+i-1},\xi _{i})$, $i=2$, $\ldots$, $d-s$,
    so it has a root there. This makes not less than $d-2$ distinct real
    roots.

    If $\tau _s>0$ (resp. $\tau _s<0$), then $Y_d$ changes sign in each of the
    intervals
    $(\eta _s,0)$ and $(\tau _s,\xi _1)$ (resp. $(\eta _s,\tau _s)$
    and $(0,\xi _1)$), so it has two more real distinct roots.
    Hence for any~$t\geq 0$, $Y_d$ is hyperbolic, with $d$ distinct roots. 
  \end{proof}

\begin{proof}[Proof of Theorem~\ref{tmconcrete}] We remind that 
  we denote by $H^k_{\pm}$ not only the graphs mentioned in
  Theorem~\ref{tmremind}, but also the corresponding functions.
  \vspace{1mm}
  
  A) We prove Theorem~\ref{tmconcrete} by induction on $d$.
  The induction base are the cases $d=2$ and $d=3$, see
  Examples~\ref{exk2} and~\ref{exk3}. 

  Suppose that Theorem~\ref{tmconcrete}
  holds true for $d=d_0\geq 3$. Set
  $d:=d_0+1$. As in the proof of Theorem~\ref{tmconnecthyp} we set
  $Q_{d}:=x^{d}+a_{d-1}x^{d-1}+\cdots +a_2x^2+a_1x+a_0$ and
  $Q_{d-1}:=x^{d-1}+a_{d-1}x^{d-2}+\cdots +a_2x+a_1$, so that the
  intersection $\Pi _d\cap \{ a_0=0\}$ can be identified with
$\Pi _{d-1}$.
\vspace{1mm}

B) We remind that any stratum (or component) $U$ of $\Pi _{d-1}^*$ is
of the form (see Notation~\ref{notaPid} and Remark~\ref{remstrata})

$$U=S_{1^{d-1}}(\sigma _1,\ldots ,\sigma _{d-1})=
\{ a'\in S_{1^{d-1}}~|~{\rm sign}(a_j)=\sigma _j,
1\leq j\leq d-1\} ~.$$
Starting with such a component $U$ (hence $U=U^{d-1}$),
we construct
in several steps the
components $U_+$ and $U_-$ of the set $\Pi _d^*$ sharing with $U$ the signs
of the coefficients $a_{d-1}$, $\ldots$, $a_1$. One has $a_0>0$ in $U_+$ and
$a_0<0$ in $U_-$.

At the first step we construct the sets $U_{1,\pm }$ as
follows. We remind that the projections $\pi ^k$ and their fibres $\tilde{f}_k$
were defined in part (1) of
Theorem~\ref{tmremind}. Each fibre $\tilde{f}_d$ of the projection $\pi ^d$
which is over a
point of $U$
is a segment, see part (1) of Theorem~\ref{tmremind}. If $Q_{d-1}\in U$, then
for $\varepsilon >0$ small enough, both polynomials $xQ_{d-1}\pm \varepsilon$
are hyperbolic. Indeed, all roots of $Q_{d-1}$ are real and simple.
The set $U_{1,+}$ (resp. $U_{1,-}$) is the union of the
interior points of these fibres $\tilde{f}_d$ which are with positive
(resp. with negative)
$a_0$-coordinates. Thus

$$\begin{array}{ll}
  U_{1,+}=\{ a\in \tilde{f}_d~|~a'\in U,~0<a_0<H^d_+(a')\} &{\rm and}\\ \\
   U_{1,-}=\{ a\in \tilde{f}_d~|~a'\in U,~H^d_-(a')<a_0<0\} ~,&\end{array}$$
see part (6) of
Theorem~\ref{tmremind}).
Hence the sets $U_{1,\pm }$ are open, non-empty and
contractible.

For $d\geq 2$, the intersection $\Pi _d\cap \{ a_0=0\}$ is strictly included
in the projection $\Pi _d^{d-1}$ of $\Pi _d$ in $Oa_1\cdots a_{d-1}$. Therefore
one can expect that the sets $U_{1,\pm}$ are not the whole of two components
of $\Pi _d^*$. We construct contractible sets
$U_{1,\pm}\subset U_{2,\pm}\subset \cdots \subset U_{d-1,\pm}$, where
for $1\leq j\leq d-1$, the signs of the coordinates $a_j$
of each point of $U_{k,+}$ (resp. $U_{k,-}$)
are defined by $\sigma$, and $U_{d-1,\pm}$ are components of $\Pi _d^*$.
One has $a_0>0$ in $U_{k,+}$ and $a_0<0$ in $U_{k,-}$.

\vspace{1mm}

C) Recall that the set $U$ consists of all the points between the graphs
$L^{d-1}_{\pm}$ of
two continuous functions defined on $U^{d-2}$:

$$U=\{ a'~|~L^{d-1}_-(a'')<a''<L^{d-1}_+(a''),~a''\in U^{d-2}\} ~,$$
see Notation~\ref{notaPid}. 
Thus
$(L^{d-1}_+\cup L^{d-1}_-)\subset \partial U$. Depending on the sign of $a_1$
in $U$, for each of these graphs, part or the whole of it could belong to
the hyperplane $a_1=0$.

Consider a fibre $\tilde{f}_d$ over a point of one of the graphs
$L^{d-1}_{\pm}$ and not belonging to the hyperplane $a_1=0$.
A priori the two endpoints of the fibre cannot have
$a_0$-coordinates with opposite signs. Indeed, if this were the case for
the fibre over $a'={a^*}'$ (see Notation~\ref{notaPid}), then
for all fibres over $a'$ close to ${a^*}'$, these signs would also
be opposite, because the functions $L^d_{\pm}$, whose values are the values of
the $a_0$-coordinates of the endpoints, are continuous. Hence all these
fibres $\tilde{f}_d$ intersect the hyperplane $a_0=0$
(see part (1) of Theorem~\ref{tmremind}),
but not the hyperplane $a_1=0$. Hence the point ${a^*}'$ is an interior point
of $\Pi _d$ (hence of $U$ as well) and not a point of $\partial U$
which is a contradiction, see
part (2) of Theorem~\ref{tmremind}.

Both endpoints cannot
have non-zero coordinates of the same sign, because
then in the same way the fibres $\tilde{f}_d$ over all points
$a'$ close to ${a^*}'$ would not intersect the hyperplane $a_0=0$ hence
${a^*}'\not\in \overline{U}$, so ${a^*}'\not\in \partial U$.

Hence the following three possibilities remain:
\vspace{1mm}

a) both endpoints have zero $a_0$-coordinates;
\vspace{1mm}

b) one endpoint has a zero and the other endpoint has a positive
$a_0$-coordinate;
\vspace{1mm}

c) one endpoint has a zero and the other endpoint has a negative
$a_0$-coordinate.
\vspace{1mm}

D) Consider the points of the graph $L^{d-1}_+$
which do not belong to the
hyperplane $a_1=0$ (for $L^{d-1}_-$ the reasoning is similar). If for
$B\in (L^{d-1}_+\setminus \{ a_1=0\})$, possibility
a) takes place, then there is nothing to do.

Suppose that possibility b) takes place. Denote by $a_{j,B}$ the coordinates
of the point $B$ (hence $a_{0,B}=0$). For each such point $B$,
fix the coordinates
$a_j=a_{j,B}$ for $j\neq 1$ and increase $a_1$. The interior points of the
corresponding fibres $\tilde{f}_d$ (when non-void) have the same signs of their
$a_0$-coordinates, hence these signs are positive. 
Then for some $a_1=a_{1,C}>a_{1,B}$, one has either $a_{1,C}=0$ (this can happen
only when $a_{1,B}<0$) or the point
$C$ belongs to the graph $H^{d-1}_+$ and for $a_1>a_{1,C}$, the fibres
$\tilde{f}_d$ are void, see Theorem~\ref{tmremind}.

In both these situations
we add to the set
$U_{1,+}$ the points of the interior of all fibres $\tilde{f}_d$ over the
interval $[a_{1,B},a_{1,C})$ (with $a_j=a_{j,B}$ for $j\neq 1$), over all points
  $B\in (L^{d-1}_+\setminus \{ a_1=0\} )$.

  If possibility c) takes place, then we fix again $a_{j,B}$ for $j\neq 1$
  and increase $a_1$. The interior points of the
  corresponding fibres $\tilde{f}_d$ (when non-void) have negative sign
  of their $a_0$-coordinates.
  We add to the set $U_{1,-}$ the interior points 
  of all fibres $\tilde{f}_d$ over the
interval $[a_{1,B},a_{1,C})$ (with $a_j=a_{j,B}$ for $j\neq 1$), over all points
  $B\in (L^{d-1}_+\setminus \{ a_1=0\} )$.
  \vspace{1mm}
  
  E) We perform a similar reasoning and construction with $L^{d-1}_-$
  (in which the role of $H^{d-1}_+$ is played by $H^{d-1}_-$). In this case
  $a_1$ is to be decreased, one
  has $a_{1,C}<a_{1,B}$ and the interval $[a_{1,B},a_{1,C})$ is to be replaced
    by the interval $(a_{1,C},a_{1,B}]$.
  \vspace{1mm}
  
  F) Thus we have
  enlarged the sets $U_{1,\pm}$; the new sets are denoted by $U_{2,\pm}$:

  $$\begin{array}{lll}
    U_{2,+}&=&U_{1,+}\cup \{ a\in \Pi ^*_d~|~a''\in U^{d-2},~
    L_+^{d-1}(a'')<a_1<H^{d-1}_+(a''),\\ \\ &&{\rm if}~L_+^{d-1}(a'')\geq 0,~
    L_+^{d-1}(a'')<a_1<{\rm min}(0,H^{d-1}_+(a'')),~{\rm if}~L_+^{d-1}(a'')<0~\} ~,
    \\ \\
    U_{2,-}&=&U_{1,-}\cup \{ a\in \Pi ^*_d~|~a''\in U^{d-2},~
    H^{d-1}_-(a'')<a_1<L_-^{d-1}(a''),\\ \\ &&{\rm if}~L_-^{d-1}(a'')\leq 0,~
    {\rm max}(0,H^{d-1}_-(a''))<a_1<L_-^{d-1}(a''),~{\rm if}~L_-^{d-1}(a'')>0~\} ~.
    \end{array}$$

  The
  sets $U_{1,\pm}$ and $U_{2,\pm}$ satisfy the conclusion of
  Theorem~\ref{tmconcrete}. We denote the graphs $L^k_{\pm}$ defined for
  the sets $U_{1,\pm}$ and $U_{2,\pm}$ by $L^k_{1,\pm}$ and $L^k_{2,\pm}$.
  The construction of these graphs implies that they are graphs of continuous
  functions (because such are the graphs $H^k_{\pm}$). 
  The set $U_{1,+}\cup U_{1,-}$ 
  (resp. $U_{2,+}\cup U_{2,-}$) contains all points of the set
  $(\pi ^d)^{-1}(U)\cap \Pi _{d,\sigma}^*$ (resp.
  $(\pi ^{d-1}\circ \pi ^d)^{-1}(U^{d-2})\cap \Pi _{d,\sigma}^*$).
  
  G) We remind that $\tilde{f}_d=f^{\diamond}_{d-1}$, see Remark~\ref{remfibre}.
  Suppose that the sets $U_{s,\pm}$, $2\leq s\leq d-3$,
  are constructed such that they satisfy the
  conclusion of Theorem~\ref{tmconcrete} (the graphs $L^k_{\pm}$
  are denoted by $L^k_{s,\pm}$) and that the set
  $U_{s,+}\cup U_{s,-}$ contains all points of the set
  $(\pi ^{d-s+1}\circ \cdots \circ \pi ^d)^{-1}(U^{d-s})\cap \Pi _{d,\sigma}^*$. 

  Consider a point $D\in L^{d-s}_+$ which
  does not belong to the hyperplane $a_s=0$. For the fibre $f^{\diamond}_{d-s}$
  of the projection $\pi ^{d-s+1}\circ \cdots \circ \pi ^d$ which is over 
  $D$ (see Remark~\ref{remfibre}) one of the three possibilities takes place:
  \vspace{1mm}
  
  a') the minimal and the maximal possible value of the $a_s$-coordinate of
  the points of the fibre are zero;
  \vspace{1mm}
  
  b') the minimal possible value is $0$ and the maximal possible value is
  positive;
  \vspace{1mm}
  
  c') the minimal possible value is negative and the maximal possible value
  is~$0$.
  \vspace{1mm}

  It is not possible to have both the maximal and minimal possible value
  of $a_s$ non-zero, because in this case the point $D$ does not belong to
  the set $\partial U^{d-s}$. This is proved by analogy with~C). With regard to
  Remark~\ref{remfibre}, when the fibre $f^{\diamond}_{d-s}$ is not a point, then
  the maximal (resp. the minimal) value of $a_s$ is attained at one of
  the $0$-dimensional cells (resp. at the other $0$-dimensional cell) and
  only there. This can be deduced from part (2) of Theorem~\ref{tmremind}. 
  \vspace{1mm}
  
  H) When possibility a') takes place, then there is nothing to do. Suppose
  that possibility b') takes place. Denote by $a_{j,D}$ the coordinates of the
  point $D$ (hence $a_{0,D}=\cdots =a_{s-1,D}=0$). Fix $a_{j,D}$ for $j\neq s$
  and increase $a_s$.
  Then for some $a_s=a_{s,E}>a_{s,D}$, one has either $a_{s,E}=0$ (which is
  possible only if $a_{s,D}<0$) 
  or the point $E$ belongs
  to the graph $H^{d-s}_+$. In this case we add to the set $U_{s,+}$ the points
  of the interior of all fibres $f^{\diamond}_{d-s}$ over the interval
  $[a_{s,D},a_{s,E})$ (with $a_j=a_{j,D}$ for $j\neq s$), over all points
    $D\in (L^{d-s}_+\setminus \{ a_s=0\} )$. The $a_{s-1}$-coordinates of
    all points thus added are positive. 

    If possibility c') takes place, then we fix again $a_{j,D}$ for $j\neq s$
  and increase $a_s$. We add to the set $U_{s,-}$ the points of the interior
  of all fibres $f^{\diamond}_{d-s}$ over the
interval $[a_{s,D},a_{s,E})$ (with $a_j=a_{j,D}$ for $j\neq s$), over all points
  $D\in L^{d-s}_+\setminus \{ a_s=0\}$. The $a_{s-1}$-coordinates of
    all points thus added are negative.

We consider in a similar way the graph $L^{d-s}_-$
in which case the role of $H^{d-s}_+$ is played by $H^{d-s}_-$, $a_s$ is
to be decreased, one
  has $a_{s,E}<a_{s,D}$ and the interval $[a_{s,D},a_{s,E})$ is to be replaced
    by the interval $(a_{s,E},a_{s,D}]$.
    \vspace{1mm}

    I) We have thus constructed the sets $U_{s+1,\pm}$ which satisfy
    the conclusion of Theorem~\ref{tmconcrete}:
    
$$\begin{array}{lll}
    U_{s+1,+}&=&U_{s,+}\cup \{ a\in \Pi ^*_d~|~a^{(s+1)}\in U^{d-s-1},\\ \\ &&
    L_{s,+}^{d-s}(a^{(s+1)})<a_s<H^{d-s}_+(a^{(s+1)}),~
    {\rm if}~L_{s,+}^{d-s}(a^{(s+1)})\geq 0,\\ \\ &&
    L_{s,+}^{d-s}(a^{(s+1)})<a_s<{\rm min}(0,H^{d-s}_+(a^{(s+1)})),~{\rm if}~L_{s,+}^{d-s}(a^{(s+1)})<0~\} ~,\\ \\
    U_{s+1,-}&=&U_{s,-}\cup \{ a\in \Pi ^*_d~|~a^{(s+1)}\in U^{d-s-1},\\ \\ &&
    H^{d-s}_-(a^{(s+1)})<a_s<L_{s,-}^{d-s}(a^{(s+1)}),~{\rm if}~
    L_{s,-}^{d-s}(a^{(s+1)})\leq 0,\\ \\ &&
    {\rm max}(0,H^{d-s}_-(a^{(s+1)}))<a_s<L_{s,-}^{d-s}(a^{(s+1)}),~{\rm if}~L_{s,-}^{d-s}(a^{(s+1)})>0~\} ~.
    \end{array}$$
    The set
    $U_{s+1,+}\cup U_{s+1,-}$ contains all points of the set
    $(\pi ^{d-s}\circ \cdots \circ \pi ^d)^{-1}(U^{d-s-1})\cap \Pi _{d,\sigma}^*$. 
    It should be noticed that as the fibres
    $f^{\diamond}_{d-s}$ contain cells of dimension from $0$ to $s$, all graphs
    $L_{s,\pm}^k$ would have to be changed when passing from
    $L_{s,\pm}^k$ to $L_{s+1,\pm}^k$. The new
    graphs are graphs of continuous functions; this follows from the
    construction and from the fact that such are the graphs $H^k_{\pm}$.
    \vspace{1mm}

    J) One can construct the sets $U_{d-1,\pm}$ in a similar way. The only
    difference is the fact that there is a graph $H^2_+$, but not a graph
    $H^2_-$, see Example~\ref{exk2}:

$$\begin{array}{lll}
    U_{d-1,+}&=&U_{d-2,+}\cup \{ a\in \Pi ^*_d~|~a^{(d-1)}\in U^{1},\\ \\ &&
    L_{d-2,+}^{2}(a^{(d-1)})<a_{d-2}<H^{2}_+(a^{(d-1)}),~
    {\rm if}~L_{d-2,+}^{2}(a^{(d-1)})\geq 0,\\ \\ &&
    L_{d-2,+}^{2}(a^{(d-1)})<a_{d-2}<{\rm min}(0,H^{2}_+(a^{(d-1)})),~{\rm if}~L_{d-2,+}^{2}(a^{(d-1)})<0~\} ~,\\ \\
    U_{d-1,-}&=&U_{d-2,-}\cup \{ a\in \Pi ^*_d~|~a^{(d-1)}\in U^{1},\\ \\ &&
    a_{d-2}<L_{d-2,-}^{2}(a^{(d-1)}),~{\rm if}~
    L_{d-2,-}^{2}(a^{(d-1)})\leq 0,\\ \\ &&
    0<a_{d-2}<L_{d-2,-}^{2}(a^{(d-1)}),~{\rm if}~L_{d-2,-}^{2}(a^{(d-1)})>0~\} ~.
    \end{array}$$
    We set $U_{\pm}:=U_{d-1,\pm}$. The set
    $U_+\cup U_-$ contains all points from the set
    $(\pi ^2\circ \cdots \circ \pi ^d)^{-1}(U^1)\cap \Pi _{d,\sigma}^*$. 
    The sets $U_{\pm}$ satisfy the conclusion of
    Theorem~\ref{tmconcrete}. Hence they are contractible.
    \vspace{1mm}

    K) The functions $L^k_{\pm}$ encountered throughout
    the proof of the theorem  
    can be extended by continuity on the closures of the sets on which they
    are defined, because this is the case of the functions $H^k_{\pm}$.
    Moreover, fibres $\tilde{f}_k$ which are points appear only
    in case they are over points of the graphs $H^{k-1}_{\pm}$. Hence this
    describes the only possibility for the values of the functions
    $L^k_{\pm}$ to coincide.
\end{proof}

\section{Comments and open problems\protect\label{seccomments}}

One could try to
generalize Theorem~\ref{tmconnecthyp} by considering instead of the
set $\Pi _d^*$ the set $R_{3,d}$, i.~e. by dropping the requirement
the polynomial
$Q_d$ to be hyperbolic. So an open problem can be formulated
like this:

\begin{op}\label{op1}
  For a given degree $d$, consider the triples $(\sigma ,~pos,~neg)$
  compatible with Descartes' rule of
  signs. Is it true that for each such triple, the corresponding subset of the
  set $R_{3,d}$ is either contractible or empty?
\end{op}

The difference between this open problem and Theorem~\ref{tmconnecthyp} is
the necessity to check whether the subset is empty or not (see part (3)
of Theorem~\ref{tm2roots}). For instance,
if $d=4$, then for neither of the triples

$$((+,-,+,+),~2,~0)~~~\, \, {\rm and}~~~\, \, ((-,-,-,+),~0,~2)$$
(both compatible with Descartes' rule of signs) does there exist a polynoial
$x^4+a_3x^3+a_2x^2+a_1x+a_0$ with signs of the coefficients $a_j$ as defined
by $\sigma$ and with $2$ positive and $0$ negative or with $0$ positive and $2$
negative roots respectively, see~\cite{Gr} (all roots are assumed to be simple).

The question of realizability of triples $(\sigma ,~pos,~neg)$ has
been asked in~\cite{AJS}. The exhaustive answer to this question is known
for $d\leq 8$. For $d=4$, it is due to D.~Grabiner (\cite{Gr}), for $d=5$ and
$6$, to A.~Albouy and Y.~Fu (\cite{AlFu}), for $d=7$ and partially for $d=8$,
to J.~Forsg{\aa}rd, V.~P.~Kostov and B.~Shapiro (\cite{FoKoSh} and
\cite{FoKoSh1})
and for $d=8$ the result was completed in~\cite{KoCzMJ}. Other results in
this direction can be found in \cite{CGK}, \cite{CGK2} and~\cite{KoMB}.

\begin{rems}
  {\rm (1) It is not easy to imagine how one could prove that all
    components of $R_{3,d}$ are either contractible or
    empty 
    without giving an exhaustive answer to the question which triples
    $(\sigma ,~pos,~neg)$ are realizable and which are not. Unfortunately,
    at present, giving such an answer for any degree $d$ is out of reach.

    (2) If one can prove not contractibility of the non-empty
    components, but only that they are (simply) connected, would
    also be of interest.}
  \end{rems}

For a degree $d$ univariate real monic polynomial $Q_d$ without
vanishing coefficients, one can define
the couples $(pos_{\ell},~neg_{\ell})$ of the
numbers of positive and negative roots
of $Q_d^{(\ell )}$, $\ell =0$, $1$, $\ldots$, $d-1$. One can observe that
the $d$ couples $(pos_{\ell},~neg_{\ell})$ define the signs of the coefficients
of $Q_d$ and that their choice must be compatible not only with Descartes'
rule of
signs, but also with Rolle's theorem. We call such $d$-tuples of couples
{\em compatible} for short. We assume that for
$\ell =0$, $1$, $\ldots$, $d-1$, all
real roots of $Q_d^{(\ell )}$ are simple and non-zero. 

To have a geometric
idea of the situation we define the discriminant sets $\tilde{\Delta}_j$,
$j=1$, $\ldots$, $d$ as the sets $\Delta _j$ defined in the spaces 
$Oa_{d-j}\ldots a_{d-1}$ for the polynomials $Q_d^{(d-j)}$. In particular,
$\tilde{\Delta}_d=\Delta _d$. For $j=1$, $\ldots$, $d-1$, we set
$\Delta _j:=\tilde{\Delta}_j\times Oa_0\ldots a_{d-j-1}$. We define the
set $R_{4,d}$ as

$$R_{4,d}:=\mathbb{R}^d\setminus \left( (\cup _{j=1}^d\Delta _j)\cup
(\cup _{j=0}^{d-1}\{ a_j=0\} )\right) ~.$$ 
For $d\leq 5$, the question when a subset of $R_{4,d}$ defined by a given
compatible $d$-tuple of couples $(pos_{\ell},~neg_{\ell})$ is empty is considered
in~\cite{CGK1}.

\begin{op}
  Given the $d$ compatible couples $(pos_{\ell},~neg_{\ell})$, is it true that
  the subset of $R_{4,d}$ defined by them is either connected (eventually
  contractible) or empty? In other words, is it true that each $d$-tuple of
  such couples defines 
  either exactly one or none of the components of the set $R_{4,d}$?
\end{op}

Some problems connected with comparing the moduli of the positive and
negative roots of hyperbolic polynomials are treated in \cite{KoPuMaDe},
\cite{KorigMO} and~\cite{KoSe}. Other problems concerning hyperbolic
polynomials are to be found in~\cite{Ko}. A tropical analog of
Descartes' rule of signs is discussed in~\cite{FoNoSh}.
\vspace{1mm}

{\bf Acknowledgement.} B. Z. Shapiro from the University of Stockholm
attracted the author's attention to Open Problem~\ref{op1} and
suggested the proof of part (1) of Theorem~\ref{tm2roots}. The remarks
of the anonymous referee allowed to improve the clarity of the proofs
of the theorems.

\end{document}